\documentclass[11pt,a4paper]{amsart}
\usepackage{amssymb,amsmath,epsfig,graphics,mathrsfs,enumerate,verbatim}
\usepackage[pagebackref,colorlinks=true,linkcolor=blue,citecolor=blue]{hyperref}
\usepackage{fancyhdr}
\usepackage{hyperref}
\hypersetup{
 colorlinks   = true,
 urlcolor     = blue,
 linkcolor    = blue,
 citecolor   = red ,
 bookmarksopen=true
}

\usepackage{amsmath}
\usepackage{amsfonts}
\usepackage{amssymb}
\usepackage{amsthm}
\usepackage{epsfig,graphics,mathrsfs}
\usepackage{graphicx}
\usepackage{dsfont}

\usepackage[usenames, dvipsnames]{color}

\usepackage{hyperref}

\def \phi {\varphi}

\def \R {\mathbb{R}}

\def \vf{\varphi}

\def \So {\mathscr{S}(\Rm)}

\newcommand{\Rn}{\mathbb R^n}
\newcommand{\Rm}{\mathbb R^m}

\newcommand{\p}{\partial}

\numberwithin{equation}{section}

\newcommand{\beq}{\begin{equation}}
\newcommand{\bea}[1]{\begin{array}{#1} }
\newcommand{\eeq}{ \end{equation}}
\newcommand{\ea}{ \end{array}}

\newcommand{\I}{\mathscr I_{HL}}

\newcommand{\tr}{\operatorname{tr}}

\newcommand{\sa}{\langle}
\newcommand{\da}{\rangle}

\newtheorem{theorem}{Theorem}[section]
\newtheorem{lemma}[theorem]{Lemma}
\newtheorem{proposition}[theorem]{Proposition}

\numberwithin{equation}{section}

\begin{document}

\title[Schr\"odinger semigroups, etc.]{Schr\"odinger semigroups and the H\"ormander hypoellipticity condition}



\date{}

\begin{abstract}
We introduce a class of (possibly) degenerate dispersive equations with a drift. We prove that, under the H\"ormander hypoellipticity condition, the relevant Cauchy problem can be uniquely solved in the Schwartz class, and the solution operator can be uniquely extended to a strongly continuous semigroup $\{\mathcal T(t)\}_{t\ge 0}$ in $L^2(\Rm)$. Finally, we prove that for $t>0$ the operator $\mathcal T(t)$ satisfies a sharp form of dispersive estimate in $L^p$, for any $1\le p\le 2$, and an uncertainty principle. 
\end{abstract}

\author{Nicola Garofalo}

\address{Dipartimento d'Ingegneria Civile e Ambientale (DICEA)\\ Universit\`a di Padova\\ Via Marzolo, 9 - 35131 Padova,  Italy}
\vskip 0.2in
\email{nicola.garofalo@unipd.it}

\thanks{N. Garofalo is supported in part by a Progetto SID (Investimento Strategico di Dipartimento): ``Aspects of nonlocal operators via fine properties of heat kernels", University of Padova (2022); and by a PRIN (Progetto di Ricerca di Rilevante Interesse Nazionale) (2022): ``Variational and analytical aspects of geometric PDEs". He is also partially supported by a Visiting Professorship at the Arizona State University}

\author{Alessandra Lunardi}

\address{Dipartimento di Scienze Matematiche, Fisiche e Informatiche
Universit\`a di Parma\\
Parco Area delle Scienze 53/A - 43124 Parma, Italy}
\vskip 0.2in
\email{alessandra.lunardi@unipr.it}

\maketitle


\section{Introduction}\label{S:intro}

In the present work we are interested in the Cauchy problem 
in $\Rm\times (0,\infty)$ for the following Schr\"odinger equation with friction
\begin{equation}\label{A}
\begin{cases}
\p_t f- i \operatorname{tr}(Q \nabla^2 f) -  \langle B x,\nabla f\rangle  = 0,
\\
f(x,0) = \vf(x),\ \ \ \ \ \vf\in \mathscr S(\Rm),
\end{cases}
\end{equation} 
where as usual $i^2 = - 1$. Here, for $m\in \mathbb N$, we let  $Q, B\in \mathbb M_{m\times m}(\R)$, with $Q = Q^\star \ge 0$. Since the matrix $Q$ can be highly singular, solving \eqref{A} in any reasonable functional space is hopeless in general. We make the key assumption that for at least one $t>0$ the \emph{covariance matrix} $Q(t)$,  generated by $Q$ and $B$, satisfy the  condition introduced by H\"ormander in his fundamental paper \cite{Ho}
\begin{equation}\label{Ds}
Q(t) = \int_0^t e^{sB}Q e^{sB^\star} ds > 0.
\end{equation}
We mention that \eqref{Ds} is equivalent to the Kalman rank condition, see \cite[Theorem 1.2, p. 17]{Z}. It is easy to see that if \eqref{Ds} holds for one $t$, it must be true for all $t>0$.

Whereas no hypoellipticity can be expected for \eqref{A} even if $Q>0$ and $B = O_m$, our objective is to prove that, under the assumption \eqref{Ds}, one can uniquely solve the Cauchy problem and the relevant solution operator $\mathcal T(t) \vf(x) = f(x,t)$ admits the following representation on $\So$ 
\begin{equation}\label{f}
\mathcal T(t) \vf(x)  = \begin{cases}
(4\pi
)^{-\frac{m}{2}} \frac{e^{-\frac{i \pi m}4}}{\sqrt{\det Q(t)}}   \int_{\Rm} e^{i \frac{\sa Q(t)^{-1}(y-e^{tB}x),y-e^{tB}x\da}{4}} \vf(y) dy,\ \ \ \ t>0,
\\
\\
\vf(x),\ \ \ \ \ \ \ \ \ \ \  t = 0,
\end{cases}
\end{equation}
see Proposition \ref{P:erwin}. We note that \eqref{A} is invariant with respect to the non-Abelian group law
$(x,s)\circ (y,t) = (y+ e^{-tB}x,s+t)$, and this is reflected in \eqref{f}.
We then show that $\{\mathcal T(t)\}_{t\ge 0}$ is a semigroup in $\So$ which can be uniquely extended to a strongly continuous one in $L^2(\Rm)$, see Theorem \ref{T:sgL2}. We prove that for any $1\le p \le 2$ and $t>0$, this family of linear operators satisfies a sharp form of dispersive estimate in $L^p(\Rm)$, see Theorem \ref{T:fest}. The latter states that, given $p$ in such range, for $t>0$ we can extend \eqref{f} to a  bounded operator $\mathcal T(t) : L^p(\Rm)\to L^{p'}(\Rm)$ such that for any $\vf\in L^p(\Rm)$ one has
\begin{equation}\label{fe0}
||\mathcal T(t) \vf||_{L^{p'}(\Rm)} \le (4\pi
)^{-\frac{m}{2}+\frac{m}{p'}} \left(\frac{p^{1/p}}{{p'}^{1/p'}}\right)^{\frac m2} \ \frac{e^{-\frac{\tr B}{p'} t}}{(\det Q(t))^{\frac 12 - \frac 1{p'}}}\ ||\vf||_{L^p(\Rm)},
\end{equation} 
where as usual, $\frac 1p + \frac 1{p'} = 1$. Finally, exploiting \eqref{f} and a classical theorem of Hardy, we prove the following form of uncertainty principle, see Theorem \ref{T:hardy}. Let $\vf\in L^2(\Rm)$ and $f(x,t)= \mathcal T(t) \vf(x)$ satisfy for some $A, a, b, s>0$ and a.e. $x\in \Rm$
\begin{equation}\label{upperG0}
|f(x,0)| \le A e^{-a |x|^2},\ \ \ \ \ |f(x,s)| \le A e^{-b |K(s)^{-1} x|^2}.
\end{equation}
If $s > \frac{\pi}{\sqrt{ab}}$,
then $f(x,t) \equiv 0$ in $\Rm\times (0,\infty)$. In the second inequality in \eqref{upperG0}, for $t>0$ the invertible linear mapping $K(t):\Rm\to \Rm$ is defined by $K(t) x = 4\pi t^{-1} Q(t) e^{-tB} x$.

To put our results in context, we recall that in the opening of his cited work \cite{Ho} H\"ormander considered the following class of evolution equations in $\Rm\times (0,\infty)$ 
\begin{equation}\label{A0}
\p_t u - \operatorname{tr}(Q \nabla^2 u) -
\langle Bx,\nabla u\rangle   = 0,
\end{equation}
and proved that \eqref{A0} is hypoelliptic if and
only if \eqref{Ds} holds. We recall that if for given $A, B \in \mathbb M_{m\times m}(\R)$ we let $Q = \frac 12 A A^\star$, then \eqref{A0} is the Kolmogorov equation of the stochastic differential equation
\[
d X_t = B X_t dt + A\ d W_t,
\]
where $W_t$ is a standard Brownian motion in $\Rm$.

The work of H\"ormander also put in a broader perspective an interesting result from thirty years earlier that, up to that point, had remained isolated, see \cite{Str}. 
In his 1934 note \cite{Kol} on Brownian motion and the kinetic theory of gases Kolmogorov considered the following special case of \eqref{A0} in $\Rm\times (0,\infty)$, with $m = 2n$ and spatial variables $v, x\in \Rn$, 
\begin{equation}\label{K}
\p_t u- \Delta_v u - \sa v,\nabla_x u\da = 0.
\end{equation}
Despite the highly degenerate nature of this PDE (note the missing diffusive term $\Delta_x u$), Kolmogorov constructed an explicit fundamental solution $C^\infty$ outside the diagonal, thus proving the hypoellipticity of his equation. Since for \eqref{K} we have $Q = \begin{pmatrix} I_n & O_n\\ O_n & O_n\end{pmatrix}$ and $B = \begin{pmatrix} O_n & O_n\\I_n & O_n\end{pmatrix}$, a computation gives  
\begin{equation}\label{QK}
Q(t)= \begin{pmatrix} tI_n & \frac{t^2}2 I_n\\ \frac{t^2}2 I_n & \frac{t^3}3 I_n\end{pmatrix}\  \Longrightarrow\ \det Q(t) = c(n) t^{4n} = c(m) t^{2m}>0,
\end{equation}
which in view of \eqref{Ds} shows that \eqref{K} nicely fits the general framework of \cite{Ho}.
But in the context of the Cauchy problem \eqref{A}, the expression of $\det Q(t)$ in \eqref{QK} is telling us something interesting. Since  $\tr B = 0$, inserting such term in our dispersive estimate \eqref{fe0}, we presently obtain the following sharp behaviour
\begin{equation}\label{KS}
||\mathcal T(t) \vf||_{L^{p'}(\Rm)} \le \frac{C(m,p)\ }{t^{2m(\frac 12 - \frac 1{p'})}} ||\vf||_{L^p(\Rm)}
\end{equation}
for the solution $f(x,t) = \mathcal T(t) \vf(x)$ of the degenerate Schr\"odinger equation with friction in $\Rm\times (0,\infty)$
\begin{equation}\label{Skol}
\p_t f- i \Delta_v f - \sa v,\nabla_x f\da = 0.
\end{equation}
Comparing \eqref{KS} with the well-known dispersive estimate satisfied by the free Schr\"odinger semigroup $\mathcal T_0(t) \vf = e^{i t \Delta} \vf$ in $\Rm$
\begin{equation}\label{class}
||\mathcal T_0(t) \vf||_{L^{p'}(\Rm)} \le \frac{C(m,p)\ }{t^{m(\frac 12 - \frac 1{p'})}} ||\vf||_{L^p(\Rm)},
\end{equation}
see e.g. \cite[Lemma 1.2]{GV} or \cite[Proposition 2.2.3]{Caze},
we see that \eqref{Skol} displays a dispersion \emph{twice} as fast as \eqref{class}. Such faster decay as $t\to \infty$ reflects the missing term $i \Delta_x u$ of order $2$ in \eqref{Skol}. 

In a completely different direction, when $Q = I_m$ and $B = - I_m$, then the PDE in \eqref{A} is  the Schr\"odinger equation with drift
\begin{equation}\label{OU}
\p_t u - i \Delta u + \sa x,\nabla u\da = 0,
\end{equation}
whose positivity-preserving counterpart is  
the Ornstein-Uhlenbeck equation $\p_t u - \Delta u +
\langle x,\nabla u\rangle  = 0$ introduced in the famous paper \cite{OU}. In this case a computation gives
\[
Q(t) = \frac{1-e^{-2t}}{2} I_m,\ \Longrightarrow\ \det Q(t) = c_m (1-e^{-2t})^m > 0.
\]
Since $\tr B = - m$, we now obtain from \eqref{fe0}
\begin{equation}\label{feOU}
||\mathcal T(t) \vf||_{L^{p'}(\Rm)} \le C(m,p)\ \frac{e^{\frac{m}{p'} t}}{(1-e^{-2t})^{m(\frac 12 - \frac 1{p'})}} ||\vf||_{L^p(\Rm)}.
\end{equation}
Note that in such case the operator norm of $\mathcal T(t)$ is controlled by a function $C(t) \cong e^{\frac{m}{p'} t}$ which blows up exponentially, instead of decaying, as $t\to \infty$. 

We mention for the interested reader that a detailed analysis of the volume function $\det Q(t)$ was carried in \cite{G&liarMA} where it was shown that, in the regime $\operatorname{tr} B\ge 0$, it blows up at least quadratically at infinity. This is sharp since for the Kramers equation in $\R^2_{(v,x)}\times (0,\infty)$ (see \cite{Bri})
\[
\p_t u - \p_{vv} u - v\p_x u + x\p_v u = 0,
\]
for which $Q = \begin{pmatrix} 1 & 0\\ 0 & 0\end{pmatrix}$, $B = \begin{pmatrix} 0 & -1\\ 1 & 0\end{pmatrix}$, and thus $\tr B = 0$, one verifies that 
\[
\det Q(t) = \pi^2(\frac{t^2}4 +\frac 18(\cos(2t) -1))>0.
\]
Therefore \eqref{Ds} holds and $\det Q(t) \cong t^2$ as $t\to \infty$.
A further analysis  which also incorporates the regime $\tr B<0$ is contained in \cite[Proposition 2.3]{BG&liar}.


\section{The Cauchy problem}\label{S:erwin}

The  objective of this section is to establish the representation formula \eqref{f} for the solution of the Cauchy problem \eqref{A}. This will be done in Proposition \ref{P:erwin} below. After proving it, we establish some further basic properties in the space $\So$ of the one-parameter family of linear operators defined by \eqref{f}. 

We begin with a key preparatory lemma that allows to eliminate the drift from \eqref{A}. We stress that Lemma \ref{L:drift} is a simple exercise on the chain rule, and that the H\"ormander condition \eqref{Ds} plays no role whatsoever. 
We mention that an alternative approach would be to use partial Fourier transform to reduce \eqref{A} to the transport equation \eqref{L} below for the function $\hat f$. After applying the method of characteristics to solve such equation, one obtains a solution formula for $\hat f$ which needs to be further manipulated to arrive at \eqref{eccola}. Lemma \ref{L:drift} significantly simplifies this approach by completely bypassing these steps. 

\begin{lemma}\label{L:drift}
Let $Q, B\in \mathbb M_{m\times m}(\R)$, and suppose that $v$ and $f$ are connected by the relation
\begin{equation}\label{drift}
v(x,t) = f(e^{- t B} x,t).
\end{equation}
Then, $f$ is a solution of the Cauchy problem \eqref{A} if and only if $v$ solves the problem
\begin{equation}\label{Abis}
\begin{cases}
\p_t v - i \operatorname{tr}(Q'(t) \nabla^2 v) = 0,
\\
v(x,0) = \vf(x),
\end{cases}
\end{equation} 
where $Q(t)$ is the matrix defined in \eqref{Ds}.
\end{lemma}

\begin{proof}
Let $f$ be a solution to the PDE in \eqref{A}, and $v$ be defined as in \eqref{drift}. 
 It is clear from the chain rule that
\[
\p_t  v(x,t) = -   \langle B e^{-  t B}x,\nabla f(e^{-  t B} x,t)\rangle + \p_t  f(e^{- t B} x,t).
\]
On the other hand, we have from \eqref{A}
\[
\p_t  f(e^{- t B} x,t)=  i \operatorname{tr}(Q \nabla^2 f(e^{- t B}x,t)) +  \langle B e^{- t B} x,\nabla f(e^{- t B}x,t)\rangle
\] 
Substituting in the above, the term with the drift disappears and we find
\begin{equation}\label{nice}
\p_t  v(x,t) = i \operatorname{tr}(Q \nabla^2 f(e^{- t B}x,t)).
\end{equation}
If now $T = [t_{ij}] \in \mathbb M_{m\times m}(\mathbb R)$,  then the chain rule gives 
\[
\frac{\p^2 f}{\p x_i\p x_j}(e^{-tB}x,t)   = \sum_{h,k=1}^m t_{hi} t_{kj} \frac{\p^2 v}{\partial y_h\p y_k}(Te^{-tB}x,t).
\] 
If we apply this identity with $T = e^{tB}$, we find
\[
\frac{\p^2 f}{\p x_i\p x_j}(e^{-tB}x,t)   = \sum_{h,k=1}^m t_{hi} t_{kj} \frac{\p^2 v}{\partial y_h\p y_k}(x,t),
\] 
and therefore
\begin{align*}
& \operatorname{tr}(Q \nabla^2 f(e^{- t B}x,t)) = \sum_{h,k=1}^m d_{hk} \frac{\p^2 v}{\partial y_h\p y_k}(x,t),
\end{align*}
where 
\[
d_{h,k} = \sum_{i,j=1}^m q_{ij}t_{hi} t_{kj} = (TQT^\star)_{h,k}.
\]
Since by \eqref{Ds} we have 
\[
TQT^\star = e^{tB} Q e^{tB^\star} = Q'(t),
\]
we finally obtain
\begin{equation}\label{nicer}
\operatorname{tr}(Q \nabla^2 f(e^{- t B}x,t)) = \operatorname{tr}(Q'(t) \nabla^2 v(x,t)).
\end{equation}
Substituting \eqref{nicer} into \eqref{nice}, we reach the conclusion that the function $v(x,t) = f(e^{-tB}x,t)$ satisfies the Cauchy problem \eqref{Abis}.

\end{proof}

As we have mentioned above, the condition \eqref{Ds} plays no role in the proof of Lemma \ref{L:drift}, but it becomes important when one tries to solve the Cauchy problem \eqref{Abis}, see \eqref{imHor} below. This aspect is connected to the following basic result concerning the Fourier transform of complex Gaussian functions for which we refer the reader to e.g. \cite[Theorem 7.6.1]{Hobook}. The definition used in this paper is the following
\[
\mathscr F f(\xi) = \hat f(\xi) = \int_{\Rm} e^{-2\pi i\sa \xi,x\da} f(x) dx.
\]

\begin{proposition}\label{P:ftgaussian}
Let $A\in G\ell(\mathbb C,m)$ be such that $A^\star = A$ and $\Re A \ge 0$. Then 
\begin{equation}\label{gengaussi2}
\mathscr F\left(\frac{1}{\sqrt{\operatorname{det} A}}(4\pi
)^{-\frac{m}{2}} e^{- \frac{\sa A^{-1}\cdot,\cdot\da}{4}}\right)(\xi) =
e^{- 4 \pi^2  \sa A\xi,\xi\da},
\end{equation}
where $\sqrt{\operatorname{det} A}$ is the unique analytic branch such that $\sqrt{\operatorname{det} A}>0$ when $A$ is real. In particular, when $A = i A_0$, with $A_0$ real, symmetric and $>0$, then \eqref{gengaussi2} gives
\begin{equation}\label{imgaussian}
\mathscr F\left(\frac{e^{-\frac{i m\pi}4}}{\sqrt{\operatorname{det} A_0}}(4\pi
)^{-\frac{m}{2}} e^{\frac{i \sa A_0^{-1}\cdot,\cdot\da}{4}}\right)(\xi) =
e^{- 4 \pi^2 i  \sa A_0\xi,\xi\da}.
\end{equation}
\end{proposition}

We also recall the following classical fact, see \cite[Section 3.4]{Hobook}.

\begin{lemma}\label{L:complexG}
Let $A\in \mathbb M_{m\times m}(\mathbb R)$ be such that $A^\star = A$ and $A > 0$. Then, as a generalised Riemann integral, 
\begin{equation}\label{iG}
\int_{\Rm} e^{i\langle A x,x\rangle} dx =
\frac{e^{i \frac m4} \pi^{\frac{m}{2}}}{\sqrt{\det A}}.
\end{equation}
\end{lemma}

We are ready to state and prove formula \eqref{f}. Henceforth in this section and in the rest of the paper the  condition \eqref{Ds} will be in force.

\begin{proposition}\label{P:erwin}
In $\Rm\times \Rm\times (0,\infty)$ consider the kernel
\begin{equation}\label{Skolmo}
\mathcal S(x,y,t) = 
(4\pi
)^{-\frac{m}{2}} \frac{e^{-\frac{i m\pi}4}}{\sqrt{\det Q(t)}}   e^{i \frac{\sa Q(t)^{-1}(y-x),y-x\da}{4}}\ \ \ \ t>0.
\end{equation}
Given $\vf\in \mathscr S(\Rm)$, the function
\begin{equation}\label{reallyfinal}
f(x,t) = 
\int_{\Rm} \mathcal S(e^{t B} x,y,t) \vf(y) dy,\ \ \ \ \ t>0,
\end{equation}
and $f(x,0) = \vf(x)$, 
is the unique solution of the Cauchy problem \eqref{A} such that $f(\cdot,t)\in \mathscr S(\Rm)$ for every $t\ge 0$.
\end{proposition}

\begin{proof}
Given $\vf\in \So$, let $f$ solve \eqref{A}. According to Lemma \ref{L:drift}, the function $v$ defined by \eqref{drift} solves the problem \eqref{Abis}. To obtain a representation of $v$, we take a partial Fourier transform with respect to $x\in \Rm$,
\[
\hat v(\xi,t) = \int_{\Rm} e^{-2\pi i\sa \xi,x\da} v(x,t) dx.
\]
The problem \eqref{Abis} is thus changed into
\begin{equation}\label{AbisFT}
\begin{cases}
\p_t \hat v + 4\pi^2 i \sa Q'(t) \xi,\xi\da \hat v = \p_t \hat v + 4\pi^2 i \frac{d}{dt} \sa Q(t) \xi,\xi\da \hat v = 0,
\\
\hat v(\xi,0) = \hat \vf(\xi).
\end{cases}
\end{equation} 
We obtain from \eqref{AbisFT}
\begin{equation}\label{hatv}
\hat v(\xi,t) = \hat \vf(\xi) e^{-4\pi^2 i \sa Q(t) \xi,\xi\da},
\end{equation}
where in the second equality we have used the definition (but not the positivity, yet) of the covariance matrix $Q(t)$ in \eqref{Ds}. To proceed, we next want to show that, for every $t>0$, the function $\xi\to e^{-4\pi^2 i  \sa Q(t) \xi,\xi\da}$ is a Fourier transform. 
For this we use \eqref{imgaussian}  in Proposition \ref{P:ftgaussian} with $A_0 = Q(t)>0$, obtaining
\begin{equation}\label{imHor}
\mathscr F\left(\frac{e^{-\frac{i m\pi}4}}{\sqrt{\det{Q(t)}}}(4\pi
)^{-\frac{m}{2}} e^{i \frac{\sa Q(t)^{-1}\cdot,\cdot\da}{4}}\right)(\xi) =
e^{- 4 \pi^2 i \sa Q(t)\xi,\xi\da}.
\end{equation}
From \eqref{hatv} and \eqref{imHor} we obtain for $t>0$ the following representation for the solution of the Cauchy problem \eqref{Abis} 
\begin{equation}\label{tildeTt}
v(x,t)  = (4\pi
)^{-\frac{m}{2}} \frac{e^{-\frac{i m\pi}4}}{\sqrt{\det Q(t)}} \int_{\Rm}  e^{i \frac{\sa Q(t)^{-1}(x-y),x-y\da}{4}} \vf(y) dy = \int_{\Rm} \mathcal S(x,y,t) \vf(y) dy,
\end{equation}
where in the second equality we have used \eqref{Skolmo}.
Combining \eqref{drift} with \eqref{tildeTt}, we finally obtain \eqref{f}, or equivalently \eqref{reallyfinal}. 

To complete the proof, and also for later purposes, we prove a representation of the spatial Fourier transform of $f(x,t) = \mathcal T(t) \vf(x)$, were $\mathcal T(t)$ is defined by \eqref{f}. With this in mind, we recall the following well-known property. If $g\in L^1(\Rm)$ and $A\in G\ell(\R,m)$, then one has
\begin{equation}\label{ftA}
\hat g(A^\star \xi) = |\operatorname{det} A|^{-1} \widehat{g \circ A^{-1}}(\xi).
\end{equation} 
Applying \eqref{ftA} with $A = e^{-tB}$, and keeping in mind that $\det A = e^{-t \tr B}$, we obtain
\[
\mathscr F(v(e^{t B} \cdot,t))(\xi) = e^{-t \tr B} \hat v(e^{-tB^\star}\xi).
\]
From this identity and from  \eqref{f}, we obtain the following basic representation formula
\begin{equation}\label{eccola}
\widehat{f(\cdot,t)}(\xi) = e^{-t \tr B} \hat \vf(e^{-tB^\star}\xi) \  e^{-4\pi^2 i  \sa Q(t) e^{-tB^\star}\xi,e^{-tB^\star}\xi\da}.
\end{equation}
Since it is clear that 
\[
\xi\ \longrightarrow\ \hat \vf(e^{-tB^\star}\xi) \in \mathscr S(\Rm),
\]
and one easily verifies that the $C^\infty$ function 
\[
\xi\ \longrightarrow\  e^{-4\pi^2 i  \sa Q(t) e^{-tB^\star}\xi,e^{-tB^\star}\xi\da}
\]
is a multiplier for $\mathscr S(\Rm)$, we conclude from \eqref{eccola} that $\widehat{f(\cdot,t)}\in \mathscr S(\Rm)$ for every $t\ge 0$, and therefore $f(\cdot,t) \in \mathscr S(\Rm)$. From these considerations, and from the (almost) obvious consequence $\widehat{f(\cdot,t)}\to \hat \vf$ of \eqref{eccola}, we conclude that 
\[
f(\cdot,t)\ \underset{t\to 0^+}{\longrightarrow}\ \vf.
\]
To complete the proof of Proposition \ref{P:erwin}, we are left with showing that \eqref{f} solves the PDE in \eqref{A}. Since we know that the function $f$ in \eqref{f} satisfies $f(\cdot,t)\in \So$, the fact that it is a solution to the PDE in \eqref{A} follows from tracing back the steps that lead to \eqref{f}. Finally, its uniqueness is again a direct consequence of \eqref{eccola}.

\end{proof}

Given a function $\vf\in \mathscr S(\Rm)$, we henceforth indicate 
\begin{equation}\label{sL}
\mathscr L \vf = i \operatorname{tr}(Q \nabla^2 \vf) +  \langle B x,\nabla \vf\rangle. 
\end{equation}
It is clear that $\mathscr L \vf\in \mathscr S(\Rm)$, or equivalently $\mathscr L(\So)\subset \So$. We explicitly note that we have shown above that \eqref{eccola} implies
\begin{equation}\label{eccolina}
\mathcal T(t) : \So\ \longrightarrow\ \So,
\end{equation}
and that
\begin{equation}\label{dt}
\frac{d}{dt} \mathcal T(t)\vf = \mathscr L \mathcal T(t) \vf.
\end{equation} 
The uniqueness part of Proposition \ref{P:erwin} implies that $\{\mathcal T(t)\}_{t\ge 0}$ is a semigroup of operators on $\So$, but it is of independent interest to also provide a direct proof, and we will do so in Proposition \ref{P:sg} below. The next result provides a basic commutation property for $\{\mathcal T(t)\}_{t\ge 0}$ that will be used more than once in this paper. 

\begin{lemma}\label{L:mapTt}
For any $\vf\in \So$ and $t>0$ we have 
\[
\mathscr L \mathcal T(t) \vf = \mathcal T(t) \mathscr L  \vf.
\]
\end{lemma}

\begin{proof}
It suffices to show that 
\begin{equation}\label{Lft}
\widehat{\mathscr L \mathcal T(t) \vf} = \widehat{\mathcal T(t) \mathscr L  \vf}.
\end{equation}
Using the well-known formulas 
\begin{equation*}
\widehat{(\p_{x_j} f)}(\xi) = 2\pi i \xi_j \hat f(\xi),\ \ \ \ \ \ \ \ \widehat{(x_j f)}(\xi) = - \frac{1}{2\pi i} \frac{\p \hat f}{\p \xi_j}(\xi),
\end{equation*}
one verifies from \eqref{sL} that
\begin{equation}\label{L}
\widehat{\mathscr L f}(\xi) = - \left\{\sa B^\star \xi, \nabla \hat f(\xi)\da + \left(4 \pi^2 i \sa Q\xi,\xi\da + \tr B\right) \hat f(\xi)\right\}.
\end{equation}
Using \eqref{L} and \eqref{eccola}, one finds
\begin{align*}
& \widehat{\mathscr L \mathcal T(t) \vf}(\xi) = - \left\{\sa B^\star \xi, \nabla \widehat{\mathcal T(t) \vf}(\xi)\da + \left(4 \pi^2 i \sa Q\xi,\xi\da + \tr B\right) \widehat{\mathcal T(t) \vf}(\xi)\right\}
\\
& = - e^{-t \tr B} \left(4 \pi^2 i \sa Q\xi,\xi\da + \tr B\right) \hat \vf(e^{-tB^\star}\xi) \  e^{-4\pi^2 i  \sa Q(t) e^{-tB^\star}\xi,e^{-tB^\star}\xi\da}
\\
& - e^{-t \tr B} \sa B^\star \xi, \nabla \left(\hat \vf(e^{-tB^\star}\xi) \  e^{-4\pi^2 i  \sa Q(t) e^{-tB^\star}\xi,e^{-tB^\star}\xi\da}\right)\da.
\end{align*}
Since
\begin{align*}
& \nabla \left(\hat \vf(e^{-tB^\star}\xi) \  e^{-4\pi^2 i  \sa Q(t) e^{-tB^\star}\xi,e^{-tB^\star}\xi\da}\right) 
\\
& = \left\{e^{-tB} \nabla \hat \vf(e^{-tB^\star}\xi) -  8\pi^2 i \hat \vf(e^{-tB^\star}\xi) e^{-tB} Q(t) e^{-tB^\star}\xi\right\} e^{-4\pi^2 i  \sa Q(t) e^{-tB^\star}\xi,e^{-tB^\star}\xi\da},
\end{align*}
we obtain
\begin{align*}
\widehat{\mathscr L \mathcal T(t) \vf}(\xi) & = - e^{-t \tr B} \bigg\{\left(4 \pi^2 i \sa Q\xi,\xi\da + 2 \sa B^\star e^{-tB^\star} \xi, Q(t) e^{-tB^\star}\xi\da+ \tr B\right) \hat \vf(e^{-tB^\star}\xi) 
\\
& + \sa B^\star e^{-tB^\star} \xi, \nabla \hat \vf(e^{-tB^\star}\xi)\da\bigg\}  e^{-4\pi^2 i  \sa Q(t) e^{-tB^\star}\xi,e^{-tB^\star}\xi\da}.
\end{align*}
On the other hand, by \eqref{eccola} and \eqref{L} we also find
\begin{align}\label{commu}
 \widehat{\mathcal T(t) \mathscr L  \vf}(\xi) & = e^{-t \tr B} \widehat{\mathscr L  \vf}(e^{-tB^\star}\xi) \  e^{-4\pi^2 i  \sa Q(t) e^{-tB^\star}\xi,e^{-tB^\star}\xi\da}
\\
& = - e^{-t \tr B} \bigg\{\sa B^\star e^{-tB^\star}\xi, \nabla \hat \vf(e^{-tB^\star}\xi)\da 
\notag\\
& + \left(4 \pi^2 i \sa Qe^{-tB^\star}\xi,e^{-tB^\star}\xi\da + \tr B\right) \hat \vf(e^{-tB^\star}\xi)\bigg\} e^{-4\pi^2 i  \sa Q(t) e^{-tB^\star}\xi,e^{-tB^\star}\xi\da}.
\notag\end{align}
Observe now that the following identity holds
\begin{equation}\label{id} 
e^{-tB}Q e^{-tB^\star}= Q  - B e^{-tB} Q(t) e^{-tB^\star} - e^{-tB} Q(t) e^{-tB^\star} B^\star,\ \ \ \ \ \ \ \ t>0.
\end{equation}
This can be verified by noting that both sides have the same value $Q$ at $t = 0$, and they have the same derivative in $t$. 
The formula \eqref{id} guarantees that
\begin{align*}
\sa Qe^{-tB^\star}\xi,e^{-tB^\star}\xi\da & = \sa Q\xi,\xi\da - \sa B Q(t)e^{-tB^\star}\xi,e^{-tB^\star}\xi\da - \sa Q(t) B^\star e^{-tB^\star}\xi,e^{-tB^\star}\xi\da
\\
& = \sa Q\xi,\xi\da - 2 \sa B^\star e^{-tB^\star} \xi, Q(t) e^{-tB^\star}\xi\da.
\end{align*}
Substituting in the above expression of  $\widehat{\mathcal T(t) \mathscr L  \vf}(\xi)$, and comparing with that of $\widehat{\mathscr L \mathcal T(t) \vf}(\xi)$, we infer that \eqref{Lft} does hold.

\end{proof}

As previously mentioned, we next intend to provide a direct proof that $\{\mathcal T(t)\}_{t\ge 0}$ is a semigroup of linear operators on $\mathscr S(\Rm)$. This will be based on the following well-known elementary but crucial algebraic property of the matrix \eqref{Ds} which incidentally establishes an important monotonicity property (in the sense of quadratic forms) of $Q(t)$.

\begin{lemma}\label{L:Q}
Let $Q(t)$ be as in \eqref{Ds}. For every $s, t >0$ one has
\[
Q(t+s) = Q(t) + e^{tB} Q(s) e^{tB^\star}.
\]
As a consequence $t\to Q(t)$ is strictly monotonically increasing in the sense of quadratic forms. 
\end{lemma}

\begin{proof}
From \eqref{Ds} we immediately obtain
\begin{align*}
Q(t+s) & = \int_0^{t+s} e^{\tau B}Q e^{\tau B^\star} d\tau
 = \int_0^{t} e^{\tau B}Q e^{\tau B^\star} d\tau + \int_t^{t+s} e^{\tau B}Q e^{\tau B^\star} d\tau
 \\
 & = Q(t) + \int_0^{s} e^{(\sigma + t) B}Q e^{(\sigma + t) B^\star} d\sigma = Q(t) + e^{tB} Q(s) e^{tB^\star}.
\end{align*}

\end{proof}

\begin{proposition}\label{P:sg}
For every $\vf\in \mathscr S(\Rm)$ and every $s, t >0$ one has
\begin{equation}\label{sgp}
\mathcal T(s+t)\vf = \mathcal T(s)(\mathcal T(t) \vf).
\end{equation}
\end{proposition}

\begin{proof}
Denoting by $\mathscr F$ the Fourier transform, it suffices to show that
\begin{equation}\label{sgFT}
\mathscr F(\mathcal T(s+t)\vf) = \mathscr F(\mathcal T(s)(\mathcal T(t) \vf)).
\end{equation}
According to \eqref{eccola}, the left-hand side of \eqref{sgFT} is given by
\[
\mathscr F(\mathcal T(s+t)\vf)(\xi) = e^{-(s+t) \tr B} \hat \vf(e^{-(s+t)B^\star}\xi) \  e^{-4\pi^2 i  \sa Q(s+t) e^{-(s+t)B^\star}\xi,e^{-(s+t)B^\star}\xi\da}.
\]
From this formula and Lemma \ref{L:Q} we obtain
\begin{align*}
& \mathscr F(\mathcal T(s+t)\vf)(\xi) = e^{-(s+t) \tr B} \hat \vf(e^{-(s+t)B^\star}\xi) \  e^{-4\pi^2 i  \sa [Q(t) + e^{tB} Q(s) e^{tB^\star}] e^{-(s+t)B^\star}\xi,e^{-(s+t)B^\star}\xi\da}
\\
& = e^{-(s+t) \tr B} \hat \vf(e^{-tB^\star}(e^{-sB^\star}\xi)) \  e^{-4\pi^2 i  \sa Q(t)e^{-tB^\star}(e^{-sB^\star} \xi),e^{-tB^\star}(e^{-sB^\star} \xi)\da} e^{-4\pi^2 i  \sa Q(s) e^{-sB^\star}\xi,e^{-sB^\star}\xi\da}
\\
& = e^{-s \tr B} e^{-4\pi^2 i  \sa Q(s) e^{-sB^\star}\xi,e^{-sB^\star}\xi\da} \mathscr F(\mathcal T(t)\vf)(e^{-sB^\star} \xi)
\\
& = \mathscr F(\mathcal T(s)(\mathcal T(t)\vf)))(\xi),
\end{align*}
where in the last two equalities we have used twice \eqref{eccola}. This proves \eqref{sgFT}.

\end{proof}

We close this section by recording in Proposition \ref{P:one} below a basic property of the kernel in \eqref{f}. The integrals in \eqref{S1}, \eqref{S2} are generalised Riemann integrals.

\begin{proposition}\label{P:one}
For every $x\in \Rm$, and $t>0$ one has 
\begin{equation}\label{S1}
\int_{\Rm} \mathcal S(e^{tB} x,y,t) dy = 1,
\end{equation}
and
\begin{equation}\label{S2}
\int_{\Rm} \mathcal S(e^{tB}x,y,t) dx = e^{-t \tr B}.
\end{equation}
\end{proposition}

\begin{proof}
We have from \eqref{f}
\begin{align*}
& \int_{\Rm} \mathcal S(e^{t B}x,y,t) dy = 
(4\pi
)^{-\frac{m}{2}} \frac{e^{-\frac{i m\pi}4}}{\sqrt{\det Q(t)}} \int_{\Rm}  e^{i \frac{|Q(t)^{-1/2}(y-e^{t B} x)|^2}{4}} dy.
\end{align*}
The change of variable $y\to z = Q(t)^{-1/2}(y-e^{t B} x)$, for which $dy = \sqrt{\det Q(t)} dz$, now gives
\begin{align*}
& \int_{\Rm} \mathcal S(e^{t B},y,t) dy = 
(4\pi)^{-\frac{m}{2}} e^{-\frac{i m\pi}4} \int_{\Rm}  e^{i \frac{|z|^2}{4}} dz  = 1,
\end{align*}
where in the last equality we have used \eqref{iG} above. This proves \eqref{S1}.

On the other hand, if we consider \eqref{S2}, then the change of variable $x\to z = Q(t)^{-1/2}(e^{t B} x-x)$, for which $dx = e^{-t \tr B} \sqrt{\det Q(t)} dz$, gives
\begin{align*}
& \int_{\Rm} \mathcal S(e^{t B}x,y,t) dx = (4\pi)^{-\frac{m}{2}} e^{-t \tr B} e^{-\frac{i m\pi}4} \int_{\Rm}  e^{i \frac{|z|^2}{4}} dz = e^{-t \tr B},
\end{align*}
which proves \eqref{S2}.

\end{proof}


\section{The extension of $\{\mathcal T(t)\}_{t\ge 0}$ to $L^2(\Rm)$}\label{S:sg}

The main result of this section is the following.

\begin{theorem}\label{T:sgL2}
The one-parameter family $\{\mathcal T(t)\}_{t\ge 0}$ of linear operators on $\mathscr S(\Rm)$ defined by \eqref{f} can be uniquely extended to a strongly continuous semigroup on $L^2(\Rm)$. Denoting by $\mathscr L_2$, $D_2$ respectively the infinitesimal generator and domain of such semigroup, so that 
\[
D_2 = \big\{f\in L^2(\Rm)\mid \mathscr L_2 f \overset{def}{=} \underset{t\to 0^+}{\lim}\ \frac{\mathcal T(t) f - f}{t}\ \emph{exists in }\ L^2(\Rm)\big\},
\]
then $\mathscr S(\Rm)$ is a core for $\mathscr L_2$, which restricted to such core is given by
\[
\mathscr L_2 \vf(x) = i \operatorname{tr}(Q \nabla^2 \vf)(x) +  \langle B x,\nabla \vf(x)\rangle = \mathscr L \vf(x).
\]
\end{theorem}

\begin{proof}

We begin with the extension of $\mathcal T(t)$ to  $L^2(\Rm)$.  From \eqref{eccola} and Plancherel theorem, for $\vf\in \So$ we  find, after the change of variable $\eta = e^{-tB^\star}\xi$, which gives $d\xi = e^{t \tr B} d\eta$, 
\begin{equation}\label{finale}
||\mathcal T(t) \vf||_{L^2(\Rm)} = e^{-t \tr B} \left(\int_{\Rm} |\hat \vf(e^{-tB^\star}\xi)|^2 d\xi\right)^{1/2} = e^{-\frac{\tr B}2 t } ||\vf||_{L^2(\Rm)}.
\end{equation}
Since $\So$ is dense in $L^2(\Rm)$, $\mathcal T(t)$ uniquely extends to a bounded operator on $L^2(\Rm)$ with
\begin{equation}\label{iso}
||\mathcal T(t) \vf||_{L^2(\Rm)} = e^{-\frac{\tr B}2 t } ||\vf||_{L^2(\Rm)}.
\end{equation}
By \eqref{sgp} and the density of $\So$, we infer that $\{\mathcal T(t)\}_{t\ge 0}$ is a semigroup in $L^2(\Rm)$.

Let now $\vf\in \So$. Using \eqref{dt} and Lemma \ref{L:mapTt}, we obtain  
\begin{equation}\label{diffT}
\mathcal T(t) \vf(x) - \vf(x)  = \int_0^t \frac{d}{ds} \mathcal T(s) \vf(x) ds = \int_0^t \mathscr L \mathcal T(s) \vf(x) ds = \int_0^t \mathcal T(s)  \mathscr L \vf(x) ds.
\end{equation}
With $M = \max\{1,e^{\frac{|\tr B|}2}\}>0$, this gives for any $0\le t \le 1$,
\[
||\mathcal T(t) \vf - \vf||_{L^2(\Rm)} \le \int_0^t ||\mathcal T(s)  \mathscr L \vf||_{L^2(\Rm)} ds \le ||\mathscr L f||_{L^2(\Rm)} \int_0^t  e^{- s \frac{\operatorname{tr} B}2} ds \le M ||\mathscr L f||_{L^2(\Rm)}\ t,
\]
where in the second inequality we have used \eqref{iso}.
This proves 
\begin{equation}\label{sc1}
||\mathcal T(t) \vf- \vf||_{L^2(\Rm)} \le M ||\mathscr L \vf||_{L^2(\Rm)}\ t\ \underset{t \to 0^+}{\longrightarrow}\ 0,
\end{equation}
for any $\vf\in \So$. As a consequence, we also have for every $\vf\in L^2(\Rm)$ 
\begin{equation}\label{sc3}
\underset{t\to 0^+}{\lim}\ ||\mathcal T(t) \vf - \vf||_{L^2(\Rm)} = 0.
\end{equation}
This shows that $\{\mathcal T(t)\}_{t\ge 0}$ is strongly continuous in $L^2(\Rm)$. 

Next, we show that $\So\subset D_2$.
For any $\vf\in \So$ we obtain from \eqref{diffT}
\[
\frac{\mathcal T(t) \vf - \vf}t - \mathscr L \vf  = \frac 1t \int_0^t \left\{\mathcal T(s)  \mathscr L \vf - \mathscr L \vf\right\} ds.
\]
Since $\mathscr L \vf\in \So$, by Minkowski's integral inequality and  \eqref{sc1} we find
\[
\left\|\frac{\mathcal T(t) \vf - \vf}t - \mathscr L \vf \right\|_{L^2(\Rm)} \le \frac 1t \int_0^t ||\mathcal T(s)  \mathscr L \vf - \mathscr L \vf||_{L^2(\Rm)} ds \le \frac M2 ||\mathscr L^2 \vf||_{L^2(\Rm)} \ t.
\]
This shows  that $\So \subset D_2$, and that in fact the  linear operators $\mathscr L_2$  and $\mathscr L$ coincide on $\So$. Finally, the fact that $\So$ is a core for $\mathscr L_2$ follows from \eqref{eccolina} above, and from the fact that $\So$ is dense in $L^2(\Rm)$, see \cite[Proposition 1.7, p. 39]{EN}.

\end{proof}



\section{The sharp dispersive estimate}\label{S:main}

This section is devoted to proving the following sharp dispersive estimate.

\begin{theorem}\label{T:fest}
Let $\{\mathcal T(t)\}_{t\ge 0}$ be the family of linear operators on $\So$ defined by \eqref{f}. 
\begin{itemize}
\item[(i)] If for some $t>0$, and $1\le p, q\le \infty$, $\mathcal T(t) : L^p(\Rm)\to L^{q}(\Rm)$ continuously, then  $1\le p\le 2$ and $q = p'$. 
\item[(ii)] Given $1\le p \le 2$, and $t>0$, we can uniquely extend \eqref{f} to a bounded linear operator $\mathcal T(t) : L^p(\Rm)\to L^{p'}(\Rm)$ such that for any $\vf\in L^p(\Rm)$ one has
\begin{equation}\label{fe}
||\mathcal T(t) \vf||_{L^{p'}(\Rm)} \le (4\pi
)^{-\frac{m}{2}+\frac{m}{p'}} \left(\frac{p^{1/p}}{{p'}^{1/p'}}\right)^{\frac m2} \ \frac{e^{-\frac{\tr B}{p'} t}}{(\det Q(t))^{\frac 12 - \frac 1{p'}}} ||\vf||_{L^p(\Rm)}.
\end{equation}
\end{itemize}
\end{theorem}

At  the heart of our proof is the key observation, which rests on the H\"ormander condition \eqref{Ds}, that for $t>0$ the operator $\mathcal T(t)$ is the Fourier transform in disguise.
If one believes this claim, the proof is pretty much finished: Theorem \ref{T:fest} is hidden in the inequality of Hausdorff-Young. The rest of this section is devoted to unravel this fact.
 
Although it will not be used in the proof of Theorem \ref{T:fest}, it is of interest to remark separately the following basic consequence of Proposition \ref{P:erwin}.

\begin{proposition}\label{P:TtL1infty}
For every $t>0$ the linear operator $\mathcal T(t)$ maps $L^1(\Rm)$ to $L^\infty(\Rm)$. In fact, for $\vf\in \mathscr S(\Rm)$ one has
\begin{equation}\label{1infty} 
||\mathcal T(t) \vf||_{L^\infty(\Rm)} \le  \frac{(4\pi 
)^{-\frac{m}{2}}}{\sqrt{\det Q(t)}} ||\vf||_{L^1(\Rm)}.
\end{equation}
\end{proposition}

We turn to the
  
\begin{proof}[Proof of Theorem \ref{T:fest}]
With \eqref{iso} in Theorem \ref{T:sgL2} and \ref{P:TtL1infty} in hands, the first idea is to appeal to the Riesz-Thorin theorem of complex interpolation, see e.g. \cite[Theorem 1.3, p. 179]{SW}. Let $1\le p\le 2$, and write 
\[
\frac 1p = \frac{1-\theta}1 + \frac{\theta}2 = 1 - \frac{\theta}2\ \ \ \ \Longrightarrow\ \ \  \frac{\theta}2 = \frac{1}{p'}.
\] 
By \eqref{iso}, \eqref{1infty}, we infer by interpolation that
\[
\mathcal T(t) : L^p(\Rm) \to L^{p'}(\Rm),
\]
and moreover
\begin{align}\label{nonsharp}
||\mathcal T(t)||_{p\to p'} & \le ||\mathcal T(t)||_{1\to \infty}^{1-\theta} ||\mathcal T(t)||_{2\to 2}^{\theta} = \left\{\frac{(4\pi 
)^{-\frac{m}{2}}}{\sqrt{\det Q(t)}}\right\}^{1-\frac{2}{p'}} \left(e^{-\frac{\tr B}2 t }\right)^{\frac{2}{p'}} 
\\
& = (4\pi
)^{-\frac{m}{2}+\frac{m}{p'}} \frac{e^{-\frac{\tr B}{p'} t}}{\det Q(t)^{(\frac 12 - \frac 1{p'})}}.
\notag
\end{align}
We are done, except that since $\left(\frac{p^{1/p}}{{p'}^{1/p'}}\right)^{\frac m2}<1$, in \eqref{fe} we have claimed a constant strictly smaller than that in \eqref{nonsharp}. So, something is missing.

To understand this point, for $t>0$ we reformulate the representation \eqref{f} in the following fashion
\begin{equation}\label{Ttnice}
\mathcal T(t) \vf(x) = (4\pi
)^{-\frac{m}{2}}  \frac{e^{-\frac{i \pi m}4}}{\sqrt{\det Q(t)}}  \int_{\Rm} e^{i \frac{|Q(t)^{-1/2}y-Q(t)^{-1/2}e^{tB}x|^2}{4}} \vf(y) dy. 
\end{equation}
Expanding the imaginary Gaussian under the integral sign, we thus see that  \eqref{Ttnice} reveals the following remarkable aspect
\begin{equation}\label{disguise}
\mathcal T(t) \vf(x) = (4\pi
)^{-\frac{m}{2}}  \frac{e^{-\frac{i \pi m}4}}{\sqrt{\det Q(t)}} e^{i \frac{|Q(t)^{-1/2} e^{tB} x|^2}{4}}\mathscr F\left(\vf e^{i \frac{|Q(t)^{-1/2} \cdot|^2}{4}}\right)((4\pi)^{-1} Q(t)^{-1} e^{tB} x).
\end{equation}
Note that by the condition \eqref{Ds}, when $t>0$ the mapping
\begin{equation}\label{dif}
x\ \longrightarrow\ Q(t)^{-1} e^{tB} x
\end{equation}
is a global diffeomorphism of $\Rm$ onto itself. Observe also that for every $t>0$ the mapping 
\[
\vf(x)\ \longrightarrow\ \vf(x) e^{i \frac{|Q(t)^{-1/2} x |^2}{4}}
\]
is a bijection on $\mathscr S(\Rm)$ (and it is also unitary on $L^p(\Rm)$ for every $1\le p\le \infty$). We thus find from \eqref{disguise}
\begin{equation*}
\frac{(4\pi
)^{-\frac{m}{2}}}{\sqrt{\det Q(t)}}\left|\mathscr F\vf(x)\right| = \left|\mathcal T(t)\left(\vf e^{-i \frac{|Q(t)^{-1/2} \cdot|^2}{4}}\right)(4\pi Q(t) e^{-tB} x)\right|.
\end{equation*}
This shows that, in modulus, and up to conjugation with imaginary Gaussians and composition with the mapping \eqref{dif}, for every $t>0$ the linear mapping $\mathcal T(t):\So\to \So$ is in disguise the Fourier transform $\mathscr F$. Therefore, for every $t>0$ the mapping properties of $\mathcal T(t)$ are decided by the  theorem of Hausdorff-Young. Since the latter states that if $\mathscr F : L^p(\Rm)\to L^q(\Rm)$, then we must necessarily have $q = p'$, and $1\le p \le 2$, we conclude that the same holds for every linear operator $\mathcal T(t)$. This proves part (i).  

For part (ii) we recall that in his celebrated paper \cite{Be} Beckner  computed the sharp constant in the Hausdorff-Young inequality, and proved that 
\begin{equation}\label{HY}
||\mathscr F \vf||_{L^{p'}(\Rm)} \le \left(\frac{p^{1/p}}{{p'}^{1/p'}}\right)^{\frac m2} ||\vf||_{L^{p}(\Rm)}.
\end{equation}
He also showed that equality is attained in \eqref{HY} if and only if $\vf$ is a Gaussian.
Going back to \eqref{disguise}, we thus have for $t>0$ and any $\vf\in \mathscr S(\Rm)$
\begin{align*}
||\mathcal T(t) \vf||_{L^{p'}(\Rm)} = \frac{(4\pi
)^{-\frac{m}{2}}}{\sqrt{\det Q(t)}}\left(\int_{\Rm} \left|\mathscr F\left(\vf e^{i \frac{|Q(t)^{-1/2} \cdot|^2}{4}}\right)((4\pi)^{-1} Q(t)^{-1} e^{tB} x)\right|^{p'} dx\right)^{1/p'}.
\end{align*}
The change of variable $y = (4\pi)^{-1} Q(t)^{-1} e^{tB} x$, which gives 
\[
dx = (4\pi)^m e^{-t \tr B} \det Q(t)\ d y,
\]
yields
\begin{align*}
||\mathcal T(t) \vf||_{L^{p'}(\Rm)} & = \frac{(4\pi
)^{-\frac{m}{2}+\frac{m}{p'}} e^{-t {\frac{\tr B}{p'}}}}{(\det Q(t))^{\frac 12 - \frac 1{p'}}}\left(\int_{\Rm} \left|\mathscr F\left(\vf e^{i \frac{|Q(t)^{-1/2} \cdot|^2}{4}}\right)(y)\right|^{p'} dy\right)^{1/p'}
\\
& \le \frac{\left(\frac{p^{1/p}}{{p'}^{1/p'}}\right)^{\frac m2}(4\pi
)^{-\frac{m}{2}+\frac{m}{p'}} e^{-t {\frac{\tr B}{p'}}}}{(\det Q(t))^{\frac 12 - \frac 1{p'}}}\left(\int_{\Rm} \left|\vf(y) e^{i \frac{|Q(t)^{-1/2} y|^2}{4}}\right|^{p} dy\right)^{1/p}
\notag
\\
& = \frac{\left(\frac{p^{1/p}}{{p'}^{1/p'}}\right)^{\frac m2} (4\pi
)^{-\frac{m}{2}+\frac{m}{p'}} e^{-t {\frac{\tr B}{p'}}}}{(\det Q(t))^{\frac 12 - \frac 1{p'}}}\left(\int_{\Rm} |\vf(y)|^{p} dy\right)^{1/p},
\notag
\end{align*}
where in the second line we have used \eqref{HY}. This proves \eqref{fe}.
\end{proof}

Finally, it might be of interest to compare the dispersive estimate \eqref{fe0} with the ultracontractive inequality for the (positivity-preserving) semigroup $\{\mathcal P(t)\}_{t\ge 0}$ associated with the Cauchy problem for \eqref{A0}. For the latter, one has the following result (see e.g. \cite[Proposition 2.3]{G&liarNA}): for every $1\leq p< \infty$ and $q\geq p$, one has $\mathcal P_t: L^p(\Rm) \to L^q(\Rm)$ for any $t>0$, with
\begin{equation}\label{ultra}
||\mathcal P_t f||_{L^q(\Rm)} \le C(m,p,q) \frac{e^{- \frac{\operatorname{tr} B}{q} t}}{(\det Q(t))^{\frac 12(\frac{1}{p}-\frac{1}{q})}} ||f||_{L^p(\Rm)}.
\end{equation}
If in \eqref{ultra} we take in particular $1\le p\le 2$, and $q = p'$, we obtain 
\begin{equation*}
||\mathcal P_t f||_{L^{p'}(\Rm)} \le C(m,p) \frac{e^{-  \frac{\operatorname{tr} B}{p'} t}}{(\det Q(t))^{\frac{1}{2}-\frac{1}{p'}}} ||f||_{L^{p}(\Rm)},
\end{equation*}
which at the formal level is identical to \eqref{fe0}. 


\section{Hardy uncertainty principle}

A classical uncertainty theorem by Hardy states the following, see \cite{Ha}, \cite{Veluma} and also \cite[Theorem 2]{CEKPV} for a real-variable proof. Let $f:\Rm\to \mathbb C$ be a measurable function for which there exist constants $C, a, b>0$ such that for every $x\in \Rm$
\begin{equation}\label{hardy}
|f(x)|\le C e^{- a |x|^2},\ \ \ \ \ \ \ \ |\hat f(x)| \le C e^{- b |x|^2}.
\end{equation}
If $a b > \pi^2$, then $f\equiv 0$ in $\Rm$. If instead $a b = \pi^2$, then $f(x) = c e^{- a |x|^2}$. 

It is well-known, see \cite{EKPV}, \cite{Cha}, \cite{CEKPV}, that the theorem of Hardy is equivalent to a uniqueness result for global solutions of the Cauchy problem for the free Schr\"odinger equation.
In what follows we provide yet another instance of this aspect. It will be convenient to introduce the following notation. For $t>0$ we consider the linear mapping $K(t) : \Rm\to \Rm$ defined by the equation 
\begin{equation}\label{Kt}
K(t) x = 4\pi t^{-1} Q(t) e^{-tB} x.
\end{equation}
As we have noted, since \eqref{Ds} implies that
\[
\det K(t) = (4\pi)^m t^{-m} e^{-t \tr B} \det Q(t)\ >\ 0,
\] 
we have $K(t)\in G\ell(\R,m)$ for every $t>0$. 

\begin{theorem}\label{T:hardy}
Let $\vf\in L^2(\Rm)$ and $f(x,t)= \mathcal T(t) \vf(x)$. Suppose that for some $A, a, b, s>0$ and for a.e. $x\in \Rm$
\begin{equation}\label{upperG}
|f(x,0)| = |\vf(x)| \le A e^{-a |x|^2},\ \ \ \ \ |f(x,s)| \le A e^{-b |K(s)^{-1} x|^2}.
\end{equation}
If $s > \frac{\pi}{\sqrt{ab}}$,
then $f(x,t) \equiv 0$ in $\Rm\times (0,\infty)$.
\end{theorem}

\begin{proof}
For any $t>0$ consider the function 
\[
\psi_t(y) = \vf(y)\ e^{i \frac{|Q(t)^{-1/2} y|^2}{4}}.
\] 
Using the mapping \eqref{Kt}, we can rewrite \eqref{disguise} in the following way 
\begin{equation}\label{redisguise}
\mathcal T(t) \vf(x) = (4\pi
)^{-\frac{m}{2}}  \frac{e^{-\frac{i \pi m}4}}{\sqrt{\det Q(t)}} e^{i \frac{|Q(t)^{-1/2} e^{tB} x|^2}{4}}\widehat{\psi_t}(t^{-1} K(t)^{-1} x).
\end{equation}
When $t = s$, the identity \eqref{redisguise} and the second inequality in the hypothesis \eqref{upperG}, imply
\begin{align}\label{beauty}
|\widehat{\psi_s}(s^{-1} y)| & =  (4\pi
)^{\frac{m}{2}} \sqrt{\det Q(s)}\ |\mathcal T(s) \vf(K(s) y)| = (4\pi
)^{\frac{m}{2}} \sqrt{\det Q(s)}\ |f(K(s) y,s)|
\\
& \le A  (4\pi
)^{\frac{m}{2}} \sqrt{\det Q(s)}\ e^{- b |y|^2} = C e^{-b |y|^2},\ \ \ \ \ \forall y\in \Rm,
\notag
\end{align}
with $C = A  (4\pi
)^{\frac{m}{2}} \sqrt{\det Q(s)}$. On the other hand, the above definition of $\psi_s$, and \eqref{upperG}, give
\begin{equation}\label{psis}
|\psi_s(x)| \le A e^{- a |x|^2},\ \ \ \ \ \forall x\in \Rm. 
\end{equation}
From Hardy's theorem and \eqref{beauty}, \eqref{psis}, we infer that if $s > \frac{\pi}{\sqrt{ab}}$ holds, then
\[
\psi_s \equiv 0\ \Longrightarrow\ \phi \equiv 0.
\]
We conclude that $f\equiv 0$ in $\Rm\times [0,\infty)$. 
 
\end{proof}


\bibliographystyle{amsplain}

\end{document}